\documentclass[12pt, reqno]{amsart}
\usepackage{hyperref}
\usepackage{amssymb, amsmath, amsthm}
\usepackage[mathscr]{eucal}

 \theoremstyle{definition}
 \newtheorem{defn}{Definition}
 \theoremstyle{remark}
 \newtheorem{rem}{Remark}
 \newtheorem{thrm}{Theorem}
\newtheorem{lemma}{Lemma}

\begin{document}
\title[A sufficient condition for a polyhedron to be rigid]{A sufficient condition for a polyhedron\\ to be rigid}
\author{Victor Alexandrov}
\address{Sobolev Institute of Mathematics, Koptyug ave., 4, 
Novosibirsk, 630090, Russia and Department of Physics, 
Novosibirsk State University, Pirogov str., 2, Novosibirsk, 
630090, Russia}
\email{alex@math.nsc.ru}
\address{{}\hfill{December 26, 2018}}
\begin{abstract}
We study oriented connected closed polyhedral surfaces with non-degenerate triangular faces 
in three-dimensional Euclidean space, calling them polyhedra for short.
A polyhedron is called flexible if its spatial shape can be changed continuously by 
changing its dihedral angles only.
We prove that the polyhedron is not flexible if for each of its edges the following holds true:
the length of this edge is not a linear combination with rational coefficients of the 
lengths of the remaining edges.
We prove also that if a polyhedron is flexible, then some linear combinations of its dihedral 
angles remain constant during the flex.
In this case, the coefficients of such a linear combination do not alter during the flex, 
are integers, and do not equal to zero simultaneously.
\end{abstract}
\maketitle

\section{Introduction}\label{s1}

We study oriented connected closed polyhedral surfaces with non-degenerate triangular faces in  $\mathbb{R}^3$.
For brevity, we call them \textit{polyhedra}, as is customary in the theory of flexible polyhedra.
A polyhedron is called \textit{flexible}, if its spatial form can be changed continuously by 
changing its dihedral angles only, i.\,e., due to such a continuous (in particular, continuous with 
respect to a parameter describing the process of the flex) deformation of the polyhedron, in which each 
of its faces remains congruent to itself in the process of deformation.
A continuous family of such deformations is called the \textit{flex} of the polyhedron.

We know quite a lot about flexible polyhedra, namely:

(a) they do exist (see \cite{Br97}, \cite{Co77}; see also
\cite{Le67}, \cite{Ku79}, \cite{Co80};
moreover, they can be of any genius and can even be non-orientable, see \cite{Sh15};

(b) during the flex, they necessarily keep unaltered the so-called total mean curvature, see \cite{Al85};

(c) during the flex, they necessarily keep unaltered the volume of the domain they bound (see 
\cite{Sa95}, \cite{Sa96}, \cite{Sa98}; 
see also \cite{CSW97}, \cite{Sc04});

(d) during the flex, they necessarily keep unaltered the Dehn invariants, see \cite{GI18};

(e) the notion of a flexible polyhedron can be introduced in all spaces of constant curvature 
of dimension 3 and above (see 
\cite{St00}, \cite{St06}, \cite{Ga14c}, \cite{Ga15a}, \cite{Ga16}, 
as well as in pseudo-Euclidean spaces of dimension 3 and above
(see \cite{Al03}); moreover, it is known that in many of these spaces flexible polyhedra do exist 
and possess properties similar to properties (a)--(d) (see 
\cite{Ku79}, \cite{Ga14a}, \cite{Ga14b}, \cite{Ga15b}, \cite{Ga17}).

However, still very little is known about one of the central problems of the theory of flexible 
polyhedra: how do we know if a given polyhedron is flexible or not?
Roughly, we can say that to answer this problem we have only the classical 
results that a closed convex polyhedron is necessarily nonflexible (i.\,e., rigid)  and that a 
polyhedron that does not allow nontrivial infinitesimal deformations is also rigid. 
A large number of publications is devoted to these results, of which we will indicate only a few, 
in our opinion, the most significant ones, namely  
\cite{Al05}, \cite{AR79}, \cite{De16}, \cite{Ho89}, \cite{St68}, \cite{Po73}, \cite{Sa04}.
Among recent results in this direction, we mention the following one by I.~Kh.~Sabitov \cite{Sa02}: 
given a polyhedron, he builds a certain set of polynomials and proves that if each of these polynomials 
has at least one non-zero coefficient, then the polyhedron is rigid.

In this article, we give a fundamentally new sufficient condition for a polyhedron to be rigid.
Namely, we prove that the condition ``the length of each edge is not a $\mathbb {Q}$-linear combination 
of the lengths of the remaining edges,'' implies that the polyhedron is rigid, see Theorem \ref{thrm2} 
in Section \ref{s3}.

In addition, we prove that if a polyhedron is flexible, some linear combinations of its dihedral angles 
remain constant during the flex, see Theorem \ref{thrm3} in Section \ref{s3}.
In this case, the coefficients of such a linear combination do not change during the flex, are integers, 
and not all are equal to zero.

Our arguments essentially use one auxiliary statement, proved by A.~A.~Gaifullin and L.~S.~Ignashchenko in 
\cite{GI18}, see Lemma \ref{l1} in Section \ref{s2}.

\section{Definitions, notation, and an auxiliary statement}\label{s2}

\begin{defn}\label{def1}
A non-empty finite set $K$ of simplices in Euclidean space is called a \textit{finite strongly 
connected two-dimensional simplicial complex} if the following conditions (i)--(v) are satisfied:

(i) each face of a simplex in $K$ is itself a simplex in $K$;

(ii) the intersection of any two simplices of $K$ is a face for each of these simplices;

(iii) each simplex of $K$ is contained in a 2-dimensional simplex of $K$;

(iv) every 1-dimensional simplex of the complex $K$ is contained in exactly two 2-dimensional simplices 
of $K$;

(v) for any two 2-dimensional simplices $\sigma, \tau\in K$ there is a sequence of 2-dimensional simplices 
$\sigma=\sigma_1, \sigma_2, \dots, \sigma_n =\tau$ of the complex $K$ such that the intersection of 
$\sigma_{j} \cap \sigma_{j+1}$ is a 1-dimensional simplex of $K$ for all $j=1,\dots,n-1$.

The set of all $i$-dimensional simplices of the complex $K$ is denoted by $K_i$, $i=0,1,2$.

A complex $K$ is called \textit{oriented} if each of its 2-dimensional simplices is equipped with an 
orientation and these orientations are consistent with each other. 
The latter means that if $\sigma\in K_1$ is contained both in $\tau_1\in K_2$ and $\tau_2\in K_2$, 
then $\tau_1$ and $\tau_2$ generate opposite orientations on $\sigma$.
\end{defn}

\begin{rem}\label{rem1}
The conditions (i) and (ii) mean that $K$ is a simplicial complex;
the condition (iii), in particular, means that $K$ is a two-dimensional complex;
the condition (iv) means that $K$ has no neither boundary nor branch points;
the condition (v) means that $K$ is strongly connected.
\end{rem}

\begin{defn}\label{def2}
Let $K$ be a finite strongly connected two-dimensional simplicial complex.
A mapping $P:K\to\mathbb{R}^3$, whose restriction to each simplex of $K$ is affine, is called a
\textit{polyhedron} of the combinatorial type $K$.
For brevity, the image $P(K)$ of $K$ under the action of $P$ is sometimes called a polyhedron too.
If $P:K\to\mathbb{R}^3$ and $\sigma\in K_i$ then $P(\sigma)$ is called a \textit{vertex} of $P$, if $i=0$; 
is called an \textit{edge} of $P$, if $i=1$; and is called a \textit{face} of $P$, if $i=2$.
We say that a polyhedron $P:K\to\mathbb{R}^3$ has \textit{nondegenerate faces} (or, equivalently, 
\textit {nondegenerate triangular faces}), if, for every $\sigma\in K_2$, the face $P(\sigma)$ 
is a nondegenerate 2-dimensional simplex in $\mathbb{R}^3$. 
If $K$ is an oriented complex, then a polytope $P:K\to\mathbb{R}^3$ is called \textit{oriented}.
\end{defn}

\begin{rem}\label{rem2}
A polyhedron $P:K\to\mathbb{R}^3$ may have self-intersections, i.\,e., in Definition \ref{def2}, it may occur 
that there are simplices $\sigma,\tau\in K$ such that $\sigma\cap\tau=\varnothing$ and
$P(\sigma)\cap P(\tau)\neq \varnothing$.
\end{rem}

\begin{defn}\label{def3}
A continuous family $\{P_t\}_{t\in[0,1]}$ of polyhedra $P_t:K\to\mathbb{R}^3$ is called a
\textit{flex} of a polyhedron $P:K\to\mathbb{R}^3$ if $P=P_0$ and, for every $\sigma\in K_2$ and 
$t\in (0,1]$, the face $P_t(\sigma)$ of the polyhedron $P_t$ is congruent to the face $P_0(\sigma)$
of the polyhedron $P_0$. 
A flex $\{P_t\}_{t\in[0,1]}$ is called \textit{trivial} if $P_t$ is obtained from $P_0$ by means of
an isometry of $\mathbb{R}^3$, i.\,e., if for every $t\in (0,1]$ there is an isometry 
$A_t:\mathbb{R}^3\to\mathbb{R}^3$ such that $P_t=A_t\circ P_0$.
Otherwise, the flex is called \textit{nontrivial}. 
A polyhedron $P$ is called \textit{flexible} if there is nontrivial flex of $P$ and is called \textit{rigid} 
if every flex of $P$ is trivial. 
\end{defn}

\begin{rem}\label{rem3}
Definitions \ref{def2} and \ref{def3} concern only simplicial polytopes, but this does not reduce the generality of the 
results obtained in Section \ref{s3}.
After all, even if we would use some more general definition of a flexible polyhedron, which allows 
us to consider flexible polyhedra with non-triangular faces, we would additionally triangulate the faces 
of the flexible polyhedron and obtain a flexible polyhedron in the sense of Definition \ref{def3}.
Note also that, by analogy with Definitions \ref{def1}--\ref{def3}, the concept of a flexible polyhedron can be defined in 
any spatial form (i.\,e., in a connected complete Riemannian manifold of constant curvature) of dimension 3 
and above; see, e.\,g., \cite{Ku79}, \cite{St06}, \cite{Ga14c}, \cite{Ga15a}, \cite{Ga15b}, \cite{Ga17}.
\end{rem}

\begin{defn}\label{def4}
Let $K$ be a finite strongly connected oriented two-dimensional simplicial complex and
$P:K\to\mathbb{R}^3$ be an oriented polyhedron with nondegenerate faces. 
Suppose $\sigma\in K_1$, $\tau_1, \tau_2\in K_2$ and $\sigma=\tau_1\cap\tau_2$.
Let us denote by $\boldsymbol{n}(\tau_1)$ and $\boldsymbol{n}(\tau_2)$
the unit positively oriented (according to the orientation of $K$) normals to the faces
$P(\tau_1)$ è $P(\tau_2)$ respectively.
Consider first the case when $\boldsymbol{n}(\tau_1)+\boldsymbol{n}(\tau_2)\neq 0$.
Let $\boldsymbol{x}$ be an arbitrary point of the edge $P(\sigma)$, which is not a vertex of $P(K)$, 
and let $B(\boldsymbol{x})$ be a ball in $\mathbb{R}^3$ which is centered at $\boldsymbol{x}$ and
does not intersect any edge of the faces $P(\tau_1)$ and $P(\tau_2)$, except the edge $P(\sigma)$.
The set $P(\tau_1)\cup P(\tau_2)$ divides the ball $B(\boldsymbol{x})$ into two bodies (slices).
Let us denote by $B_+(\boldsymbol{x})$ the one whose interior contains the points 
$\boldsymbol{x}+\frac12\varepsilon(\boldsymbol{n}(\tau_1)+\boldsymbol{n}(\tau_2))$,
where a positive number $\varepsilon$ is less than the radius of the ball $B(\boldsymbol{x})$.
We call the number 
$\varphi_{\sigma}=2\pi\text{Vol\,}(B_+(\boldsymbol{x}))/\text{Vol\,}(B(\boldsymbol{x}))$,
the \textit{principal value of the dihedral angle} of the polyhedron $P(K)$ at the edge $P(\sigma)$.
Here $\text{Vol\,}(\cdot)$ stands for the 3-dimensional volume of the corresponding body.
If $\boldsymbol{n}(\tau_1)+\boldsymbol{n}(\tau_2)= 0$, we put by definition $\varphi_{\sigma}=0$.
We call the set $\Phi_{\sigma}=\{ \varphi_{\sigma}+ 2\pi m \ \vert \ m\in\mathbb{Z}\}$ the 
\textit{oriented dihedral angle} at the edge $P(\sigma)$ of the polyhedron $P(K)$.
Here $\varphi_{\sigma}$ stands for the principal value of the dihedral angle at the edge $P(\sigma)$.
\end{defn}

\begin{rem}\label{rem4}
It can be shown that $\varphi_{\sigma}$ is in the range from 0 to $2\pi$ and does not depend on the 
choice of the point $\boldsymbol{x}$ on the edge $P(\sigma)$.
It can also be shown that $\varphi_{\sigma}$ is not a continuous function on the set of all oriented 
polytopes of a given combinatorial type $K$ with non-degenerate faces.
The fact is that if $\varphi_{\sigma}=0$ at the edge $P(\sigma)$ of $P:K\to\mathbb {R}^3$, then 
in any neighborhood of $P$ there exist polyhedra of combinational type $K$ with non-degenerate faces, 
for which the principal value of the dihedral angle at the corresponding edge is close to $0$ 
and there are polyhedra for which it is close to $2\pi$.
On the other hand, it can be shown that the many-valued function $\Phi_{\sigma}$ is continuous on the 
set of all oriented polytopes of a given combinatorial type $K$ with non-degenerate faces.
Therefore, it is possible to find continuous selection functions for it, see \cite{Ku68}. 
In Definition \ref{def5}, we use the existence of a suitable continuous selection function.
\end{rem}

\begin{defn}\label{def5}
Let $K$ be a finite strongly connected oriented two-dimensional simplicial complex and
let $\sigma\in K_1$. 
Let $P:K\to\mathbb{R}^3$ be an oriented polyhedron with nondegenerate faces and let  
$\{P_t\}_{t\in[0,1]}$ be a flex of $P$. 
We call a continuous selection function $\widetilde{\varphi}_{\sigma}(t)$  of the multi-valued
mapping  $t\mapsto \Phi_{\sigma}(t)$ such that $\widetilde{\varphi}_{\sigma}(0)=\varphi_{\sigma}$,
where $\varphi_{\sigma}$ is a principal value of the dihedral angle of $P$ at the edge $P(\sigma)$, 
a \textit{dihedral angle} at the edge $P(\sigma)$ during the flex $\{P_t\}_{t\in[0,1]}$.
\end{defn}

The following lemma is a very special case of Statement 3.4, proved by A.~A.~Gaifullin and 
L.~S.~Ignashchenko in their article \cite{GI18}.

\begin{lemma}\label{l1}
Let $K$ be a finite strongly connected oriented two-dimensional simplicial complex, 
$P:K\to\mathbb{R}^3$ be an oriented polyhedron with nondegenerate faces, and 
$\{P_t\}_{t\in [0,1]}$ be a flex of $P$.
Then, for every $\mathbb{Q}$-linear function, $f:\mathbb{R}\to\mathbb{R}$, the expression
$\sum_{\sigma} f(\ell_{\sigma})\widetilde{\varphi}_{\sigma}(t)$ is independent of $t$, 
i.\,e. it remains constant during the flex.
Here, summation is carried out over all 1-dimensional simplices $\sigma\in K$, while 
$\ell_{\sigma}$ and $\widetilde{\varphi}_{\sigma}(t)$ denote the length of the edge $P(\sigma)$ and
the dihedral angle of the polyhedron $P$ at the edge $P(\sigma)$ during the flex $\{P_t\}_{t\in[0,1]}$,
respectively.
\hfill$\square$
\end{lemma}

We will use Lemma \ref{l1} in Section~\ref{s3}.

\section{The main results}\label{s3}

\begin{thrm}\label{thrm1}
Let $K$ be a finite strongly connected oriented two-dimensional simplicial complex, 
$P:K\to\mathbb{R}^3$ be an oriented polyhedron with nondegenerate triangular faces, and 
$\{P_t\}_{t\in [0,1]}$ be a flex of $P$.
Suppose that a 1-dimensional simplex $\sigma_{*} \in K$ is such that the length of the edge
$P(\sigma_{*})$ is not a $\mathbb{Q}$-linear combination of the lengths of the remaining edges of $P$.
Then the dihedral angle at the edge $P(\sigma_{*})$ during the flex $\{P_t\}_{t\in [0,1]}$ 
remains unaltered, i.\,e., $\widetilde{\varphi}_{\sigma_{*}}(t)=\textnormal{const}$. 
\end{thrm}

\begin{proof}
As usual, denote by $K_1$ the set of all 1-dimensional simplices of $K$ and put
$K_1^*=K_1\diagdown\{\sigma_*\}$.
Let us denote by $\Lambda_*$ the $\mathbb{Q}$-linear span of the set 
$\{\ell_{\sigma}\}_{\sigma\in K_1^*}\subset\mathbb{R}^3$,
where $\ell_{\sigma}$ stands for the length of the edge $P(\sigma)$ of $P$.
Let $\{\lambda_1,\dots,\lambda_m\}$ denote an arbitrary $\mathbb{Q}$-basis in $\Lambda_*$.

According to the conditions of Theorem \ref{thrm1}, the set
$\{\ell_{\sigma_*},\lambda_1,\dots,\lambda_m\}$ is $\mathbb{Q}$-linearly independent.
Viewing $\mathbb{R}$ as a vector space over the field $\mathbb{Q}$, expand  
$\{\ell_{\sigma_*},\lambda_1,\dots,\lambda_m\}$ 
to a Hamel basis in this space. 
Define a $\mathbb{Q}$-linear function $f_{\sigma_{*}}:\mathbb{R}\to\mathbb{R}$
by putting for every element $\lambda$ of the above constructed Hamel base in $\mathbb{R}$
\begin{equation*}
f_{\sigma_{*}}(\lambda)=
\begin{cases}
1, \mbox{ if } \lambda= \ell_{\sigma_{*}};\\
0, \mbox{ if } \lambda\neq \ell_{\sigma_{*}}.
\end{cases}
\end{equation*}

Since $\{\lambda_1,\dots,\lambda_m\}$ is a basis in $\Lambda_*$, for every $\sigma\in K_1^*$,
there are rational numbers $\alpha_{\sigma j}$, $j=1,\dots, m$ such that 
$\ell_{\sigma}=\sum_{j=1}^m\alpha_{\sigma j}\lambda_j$.
Therefore, from the definition of the function $f_{\sigma_*}$, we obtain
for every $\sigma\in K_1^*$
\begin{equation*}
f_{\sigma_*}(\ell_{\sigma})=
f_{\sigma_*}\biggl(\sum_{j=1}^m\alpha_{\sigma j}\lambda_j\biggr)=
\sum_{j=1}^m\alpha_{\sigma j}f_{\sigma_*}(\lambda_j)=0.
\end{equation*}
Hence,
\begin{equation*}
\sum_{\sigma\in K_1}f_{\sigma_*}(\ell_{\sigma})
\widetilde{\varphi}_{\sigma}(t)=
f_{\sigma_*}(\ell_{\sigma_*})\widetilde{\varphi}_{\sigma_*}(t)+
\sum_{\sigma\in K_1^*}f_{\sigma_*}(\ell_{\sigma})
\widetilde{\varphi}_{\sigma}(t)=
\widetilde{\varphi}_{\sigma_*}(t).
\end{equation*}
However, according to Lemma \ref{l1}, the left-hand side of the last equality is independent of $t$. 
So, its right-hand side does not depend on $t$ too, i.\,e., 
$\widetilde{\varphi}_{\sigma_*}(t)=\text{const}$. 
\end{proof}

\begin{thrm}\label{thrm2}
Let $K$ be a finite strongly connected oriented two-dimensional simplicial complex and
$P:K\to\mathbb{R}^3$ be an oriented polyhedron with nondegenerate triangular faces.
Suppose that the set of the lengths of the edges of $P$ is $\mathbb{Q}$-linearly independent.
Then $P$ is rigid.
\end{thrm}

\begin{proof}
Let $\{P_t\}_{t\in [0,1]}$ be a flex of the polyhedron $P=P_0:K\to\mathbb{R}^3$.  
Since, for every $\sigma\in K_1$, the length of the edge $P(\sigma)$ of $P$ cannot be represented as
a $\mathbb{Q}$-linear combination of the lengths of remaining edges of $P$ then, by Theorem \ref{thrm1},
the dihedral angle at the edge $P(\sigma)$ during the flex $\{P_t\}_{t\in [0,1]}$ remains constant,
i.\,e., $\widetilde{\varphi}_{\sigma}(t)=\text{const}$.

Let us fix $t\in (0,1]$ and a 2-dimensional simplex $\tau\in K$.
Since $\{P_t\}_{t\in [0,1]}$ is a flex, there is an isometry $A_t:\mathbb{R}^3\to\mathbb{R}^3$,
which matches the faces $P_0(\tau)\subset P_0(K)$ and $P_t(\tau)\subset P_t(K)$ 
together with their orientation, i.\,e., such that $P_t\vert_{\tau}=(A_t\circ P_0)\vert_{\tau}$.
Since, for every $\sigma\in K_1$ (including the case $\sigma\subset\tau$), we have 
$\widetilde{\varphi}_{\sigma}(t)= \widetilde{\varphi}_{\sigma}(0)$, the equality
$P_t=A_t\circ P_0$ holds true on the union of the simplex $\tau$ and all three  2-dimensional 
simplices of $K$, adjacent to $\tau$.
Continuing in the same way and using the strong connectivity of the complex $K$, we make sure 
that the equality $P_t=A_t\circ P_0$ holds true on $K$.
Hence, the flex $\{P_t\}_{t\in [0,1]}$ is trivial.
Thus, we see that the polyhedron $P=P_0:K\to\mathbb{R}^3$ allows only trivial flexes,
i.\,e., it is rigid. 
\end{proof}

\begin{rem}\label{rem5}
For the first time, Theorem \ref{thrm2} was published without proof in \cite{GI18} as Corollary 1.3.
\end{rem}

\begin{rem}\label{rem6}
It would be very interesting to determine to what extent Theorem \ref{thrm2} holds for 
non-orientable polyhedra.
The above arguments are not suitable for this, since they rely on Lemma \ref{l1}, which is 
applicable to oriented polyhedra only.
\end{rem}

\begin{rem}\label{rem7}
Theorem \ref{thrm2} gives a sufficient condition for a polyhedron to be rigid, which is fundamentally 
different from the two classical sufficient conditions for a polyhedron to be rigid mentioned in 
Section \ref{s1} (we mean the convexity and the absence of nontrivial infinitesimal deformations), 
and from a sufficient condition for rigidity of polyhedra found by I.Kh. Sabitov in \cite {Sa02}.
Recall that, in \cite[Theorem 2]{Sa02}, for each small diagonal of a polyhedron $P:K\to\mathbb{R}^3$, 
a polynomial $q(x)$ in one real variable $x$ is constructed such that

$\bullet$ the coefficients of $q(x)$ are real numbers and depend only on $K$ and the edge lengths of $P$;

$\bullet$  the length of this small diagonal is a root of a $q(x)$.

\noindent{Obviously}, if $P$ is flexible then at least one of its small diagonals is nonconstant
and, therefore, takes infinitely many values. Hence, all the coefficients of the corresponding polynomial 
$q(x)$ are zero.
This consideration led I.~Kh.~Sabitov to the following sufficient condition for a polyhedron to be rigid: 
\textit{If for each small diagonal at least one coefficient of the corresponding polynomial $q(x)$ 
is not equal to zero then the polyhedron $P$ is rigid}, see \cite[Corollary 2]{Sa02}. 
However, the application of Sabitov's condition is complicated by the fact that the problem of finding 
the polynomial $q(x)$ in itself is very difficult.
\end{rem}

\begin{rem}\label{rem8}
Theorem \ref{thrm2} in a new way expresses the well-known fact that ``almost all closed polytopes 
are inflexible.'' The novelty is that the condition interpreted as  ``almost all'' is imposed 
on the edge lengths, while in previously known theorems of this type a similar condition was imposed 
on the coordinates of the vertices. 
To be more precise, we recall that there is a natural one-to-one mapping from 
the set of all polyhedra $P:K\to\mathbb{R}^3$ of a given combinatorial type $K$ 
onto $\mathbb{R}^{3v}$, where $v$ is the number of 0-dimensional simplices of $K$.
This mapping maps $P$ to the vector in $\mathbb{R}^{3v}$ such that the set of its coordinates is
the union of the sets of coordinates of all vertices of $P$ ordered according to a preassumed 
order on the set of 0-dimensional simplices of $K$.
The following properties of the set $\mathscr{F}$ consisting of the points of $\mathbb{R}^{3v}$ 
corresponding to flexible polyhedra $P:K\to\mathbb{R}^3$ are known:

(A)  $\mathscr{F}$ is contained in the complement to an open everywhere dense subset of $\mathbb{R}^{3v}$, 
see, e.\,g., \cite{AR79}, \cite{Sa02}, \cite{Gl75};

(B)  $\mathscr{F}$ is contained in the complement to a set which consists of the points
$\boldsymbol{x}\in\mathbb{R}^{3v}$ whose coordinates are algebraically independent over $\mathbb{Q}$
(this means that, given $\boldsymbol{x}$, there is no nonzero polyhedron with rational coefficients
for which $\boldsymbol{x}$ is a root), see, e.\,g.,  \cite{Fo87}, \cite{FW13}. 

\noindent{For} completeness, we note that properties (A) and (B) are also valid in more general situations, 
namely, for various and very wide classes of frameworks in $\mathbb{R}^n$, $n\geqslant 2$, 
which, for $n\geqslant 3$, include 1-skeletons of $(n-1)$-dimensional simplicial polyhedra, 
see, e.\,g., \cite{NS16}, \cite{CG17} and literature mentioned there.
\end{rem}

The following theorem is a generalization of Theorem \ref{thrm1}:

\begin{thrm}\label{thrm3}
Let $K$ be a finite strongly connected oriented two-dimensional simplicial complex,
$K_1$ be the set of all 1-dimensional simplices of $K$,
$P:K\to\mathbb{R}^3$ be an oriented polyhedron with nondegenerate triangular faces,
$L=\cup_{\sigma\in K_1}\{\ell_{\sigma}\}$ be the set of the lengths $\ell_{\sigma}$ 
of all edges $P(\sigma)$ of $P$, and let  $\{\lambda_1,\dots,\lambda_m\}$  be a
$\mathbb{Q}$-basis in the $\mathbb{Q}$-linear span of $L$, i.\,e., the reals
$\lambda_1,\dots,\lambda_m$ are $\mathbb{Q}$-linearly independent and, for every $\sigma\in K_1$,
there are $\alpha_{\sigma j}\in\mathbb{Q}$ such that
\begin{equation*}
\ell_{\sigma}=\sum_{j=1}^m \alpha_{\sigma j}\lambda_j.
\end{equation*}
Then, for every flex $\{ P_t\}_{t\in[0,1]}$ of $P$ and every $j=1,\dots,m$, the expression  
\begin{equation}\label{eq1}
\sum_{\sigma\in K_1} \alpha_{\sigma j}\widetilde{\varphi}_{\sigma}(t)
\end{equation}
is independent of $t$.
In \textnormal{(\ref{eq1})}, the summation is taken over all 1-dimensional simplices $\sigma$ of $K$
and $\widetilde{\varphi}_{\sigma}(t)$ stands for the dihedral angle at the edge  $P(\sigma)$ during
the flex $\{P_t\}_{t\in [0,1]}$.
\end{thrm}

\begin{proof}
Treating $\mathbb{R}$ as a vector space over the field $\mathbb{Q}$, expand the set 
$\{\lambda_1,\dots,\lambda_m\}$ to a Hamel basis in this space. 
For every $j=1,\dots, m$, define a $\mathbb{Q}$-linear function $f_j:\mathbb{R}\to\mathbb{R}$ 
by putting for an element $\lambda$ of the above constructed Hamel base in $\mathbb{R}$
\begin{equation*}
f_j(\lambda)=
\begin{cases}
1, \mbox{ if } \lambda= \lambda_j,\\
0, \mbox{ if } \lambda\neq \lambda j.
\end{cases}
\end{equation*}
Then, using (\ref{eq1}) and the definition of $f_j$, we get
\begin{multline*}
\sum_{\sigma\in K_1} f_j(\ell_{\sigma})\widetilde{\varphi}_{\sigma}(t)=
\sum_{\sigma\in K_1} f_j\biggl(\sum_{k=1}^m \alpha_{\sigma k}\lambda_k\biggr)
     \widetilde{\varphi}_{\sigma}(t)=\\
=\sum_{\sigma\in K_1}\sum_{k=1}^m \alpha_{\sigma k} f_j(\lambda_k)
     \widetilde{\varphi}_{\sigma}(t)=
\sum_{\sigma\in K_1} \alpha_{\sigma j}\widetilde{\varphi}_{\sigma}(t).
\end{multline*}
However, according to Lemma \ref{l1}, the left-hand side of the last equality does not depend on $t$. 
Hence, its right-hand side does not depend on $t$. 
\end{proof}

\begin{rem}\label{rem9}
Reducing rational numbers $\alpha_{\sigma j}$ in (\ref{eq1}) to a common denominator,
we find that some linear combination of its dihedral angles remains constant during the flex of a polyhedron.
Moreover, the coefficients of this linear combination do not change during the flex, are integers, 
and not all are equal to zero.
Thus, Theorem \ref{thrm3} clarifies the origin of linear relations with integer coefficients between 
dihedral angles of Bricard octahedra, found in \cite{Al10} 
(to be more precise, in \cite{Al10}, linear relations for dihedral angles of Bricard  octahedra of 
types I, II, and III were obtained and used in the proofs of Theorems 2, 4, and 6, respectively).
\end{rem}

\section{Acknowledgement}\label{s4}
The author is grateful to Alexander A. Gaifullin and Idzhad Kh. Sabitov 
for their comments on a preliminary version of this article.

\end{document}